\documentclass[a4paper,american,english,reqno]{amsart}
\usepackage[T1]{fontenc}
\usepackage[latin1]{inputenc}
\usepackage{amsthm}
\usepackage{graphicx}
\usepackage{setspace}
\usepackage{amssymb}
\makeatletter
\numberwithin{equation}{section} 
\numberwithin{figure}{section} 
\theoremstyle{plain}
\newtheorem{thm}{Theorem}
  \theoremstyle{remark}
  \newtheorem{rem}[thm]{Remark}
  \theoremstyle{definition}
  \newtheorem{defn}[thm]{Definition}
  \theoremstyle{plain}
  \newtheorem{cor}[thm]{Corollary}
  \theoremstyle{plain}
  \newtheorem{prop}[thm]{Proposition}
  \theoremstyle{plain}
  \newtheorem{lem}[thm]{Lemma}

\usepackage{amscd}

 \newcommand{\mathscr}{\mathcal}
\newcommand{\N}{\mathbb{N}}

\newcommand{\E}{\mathbb{E}}

\newcommand{\R}{\mathbb{R}}

\newcommand{\e}{\mathrm{e}}

\renewcommand{\P}{\mathbb{P}}
\renewcommand{\Pi}{\pi}

\usepackage{hyperref}

\@ifundefined{showcaptionsetup}{}{%
 \PassOptionsToPackage{caption=false}{subfig}}
\usepackage{subfig}
\makeatother

\usepackage{babel}

\begin{document}
\selectlanguage{english}
\title{Large deviations asymptotics for large waiting times}

\date{\today}

\author{Marc Kesseböhmer and Lidong Tang}

\address{Universität Bremen, Bibliothekstrasse 1, 28356 Bremen, Germany}

\email{mhk@math.uni-bremen.de, lidong@math.uni-bremen.de}

\urladdr{www.math.uni-bremen.de/stochdyn}

\selectlanguage{american}%


\selectlanguage{english}%
\begin{abstract}
In this paper we investigate the statistics of large waiting times
(with respect to the total waiting time) for Bernoulli processes.
We determine the corresponding rate functions explicitly and prove
a large deviations asymptotic. By this we have estabished a large
deviation principle for which the rate function is not the Legendre
transform of some free energy function. 
\end{abstract}
\maketitle

\section{Introduction and statement of Main results}

We will investigate a large deviation principle for return time processes
defined over Bernoulli processes. Let $\left(\omega_{i}\right)$ be
a Bernoulli process over the symbol space $\{1,\ldots,r\}$ with probability
vector $\left(p_{1},p_{2},\ldots,p_{r}\right)$, that is, for all
$\left(k_{1},\ldots,k_{n}\right)\in\left\{ 1,\ldots,r\right\} ^{n}$
and $n\in\N$, we have $\P\left(\omega_{i_{1}}=k_{1},\ldots,\omega_{i_{n}}=k_{n}\right)=\prod_{i=1}^{n}p_{k_{i}}$.

Now, let us fix $k\in\left\{ 1,\ldots,r\right\} $ with $p:=p_{k}\in\left(0,1\right)$
and set the \emph{first hitting time} $Z_{1}\left(\omega\right)$
to $k$ to be the first occurrence of $k$ in the sequence $\omega=\left(\omega_{i}\right)$,
i.e.\[
Z_{1}(\omega):=\inf\{n\in\mathbb{N}:n>0,\:\omega_{n}=k\}.\]
Recursively, we define the $n$\emph{-th total waiting time} $Z_{n}(\omega)$
to be the $n$-th occurrence of  the symbol $k$ in the sequence $\left(\omega_{i}\right)$,
i.e. \[
Z_{n}(\omega):=\inf\{\ell\in\mathbb{N}:\ell>Z_{n-1}\left(\omega\right),\:\omega_{\ell}=k\},\: n\geq2.\]
 Let us set $Z_{0}=0$. Then the $n$\emph{-th waiting time} $z_{n}(w)$
between two successive occurrences of the symbol $k$ after the $n$-th
occurrence of $k$ is given, for $n\in\N$, by \[
z_{n}(\omega):=Z_{n}(\omega)-Z_{n-1}(\omega).\]

It is well known that in this situation the $\left(z_{n}\right)$
are i.i.d. random variables and that \[
\mathbb{E}(z_{1})=p^{-1}.\]
 In this paper we are going to study the process $\left(z_{n+1}/Z_{n}\right)$.
Clearly, this process converges to $0$ in probability (even almost
surely). Our aim is to prove a large deviation principle for this
process allowing us to estimate the probability of large waiting times
with respect to the total waiting time. More precisely, for $q>0$
we will show that the following limit \[
\lim\limits _{n}\frac{1}{n}\log\mathbb{P}\left(\frac{z_{n+1}}{Z_{n}}>q\right)=-I(q)\]
exists and that the rate function $I$ can be determined explicitly.
Interestingly, it will turn out that the rate function $I$ is not
convex and hence $I$ is not the Legendre transform of a free energy
function (see Remark \ref{rem:A-prominent-example}). 

To state our main theorem we need to fix some notation. We say that
the two positive real sequences $\left(a_{n}\right)_{n\in\mathbb{N}}$,
$\left(b_{n}\right)_{n\in\mathbb{N}}$ are \emph{comparable}, $a_{n}\asymp b_{n}$,
with \emph{comparability constants} $c_{1}>0$ and $c_{2}>0$ if for
all $n\in\mathbb{N}$ we have $c_{1}a_{n}\leq b_{n}\leq c_{2}a_{n}$
.
\begin{thm}
[Large Deviations Asymptotic]\label{thm:AsympMain} With the notation
from above we have \[
\mathbb{P}\left(\omega:\frac{z_{n+1}(\omega)}{Z_{n}(\omega)}>q\right)\asymp\left(\frac{\left(C+1\right)^{C+1}}{C^{C}}\, p\,\left(1-p\right)^{C(1+q)+q}\right)^{n},\]
 where $C:=C_{p}\left(q\right):=\left(\left(1-p\right)^{-q-1}-1\right)^{-1}$
and the comparability constants may be chosen to be $1-p$ and $1$.\end{thm}
\begin{rem}
\label{rem:qinNSimResult}We will see in the proof of Theorem \ref{thm:AsympMain}
that for $q\in\mathbb{N}$ a stronger asymptotic holds. Namely, \[
\mathbb{P}\left(\omega:\frac{z_{n+1}(\omega)}{Z_{n}(\omega)}>q\right)\sim\left(\frac{(C+1)^{C+1}}{C^{C}}p(1-p)^{(1+q)C+q}\right)^{n},\]
where $a_{n}\sim b_{n}$ means that $\lim_{n\rightarrow\infty}a_{n}/b_{n}=1$
for some positive sequences $\left(a_{n}\right)$ and $\left(b_{n}\right)$.
\end{rem}
As an immediate consequence we obtain a large deviation deviation
principle. For this let us recall the \emph{Large Deviation Principle
}(see e.g.\ \cite{MR793553,DemboZeitouni:98}). 
\begin{defn}
\label{def:LD} Let $\mathcal{H}$ be a complete separable metric
space, $\mathcal{B}\left(\mathcal{H}\right)$ the Borel $\sigma$-algebra
of $\mathcal{H}$, and $\left(Q_{n}\right)_{n\in\mathbb{N}}$ a sequence
of Borel probability measures on $\mathcal{B}(\mathcal{H})$. We say
$\left(Q_{n}\right)$ \textit{\emph{fulfills a}}\textit{ Large Deviation
Principle} with normalising sequence $\left(a_{n}\right)_{n\in\mathbb{N}}\in\R_{+}^{\N}$
($\lim\limits _{n\rightarrow\infty}a_{n}\rightarrow\infty$) and \emph{rate
function }$I:\mathcal{H}\rightarrow[0,\infty]$ if the following properties
are fulfilled.
\begin{enumerate}
\item $I$ is lower semi-continuous on $\mathcal{H}$, 
\item $I$ has compact level sets, 
\item $\limsup\limits _{n\rightarrow\infty}a_{n}^{-1}\ln Q_{n}\{K\}\leq-\inf_{x\in K}I(x)$
for all closed sets $K\subset\mathcal{H}$, 
\item $\liminf\limits _{n\rightarrow\infty}a_{n}^{-1}\ln Q_{n}\{G\}\geq-\inf_{x\in G}I(x)$
for all open sets $G\subset\mathcal{H}$. 
\end{enumerate}
\end{defn}
\begin{cor}
\label{cor:MainTheorm}Let $Q_{n}$ be the distribution of $z_{n+1}/Z_{n}$
with respect to the Bernoulli measure $\mathbb{P}$, i.e. $Q_{n}\left(A\right)=\mathbb{P}\left(\omega:z_{n+1}(\omega)/Z_{n}(\omega)\in A\right)$.
Then $\left(Q_{n}\right)_{n\in\mathbb{N}}$ fulfills the Large Deviation
Principle with normalising sequence $\N$ and rate function $I_{p}$
given for $q>0$ by \[
I_{p}\left(q\right):=-\left(C+1\right)\log\left(C+1\right)+C\log C-\log p-\left((1+q)C+q\right)\log(1-p)\]
and, for $q\leq0$, by $I_{p}\left(q\right)=+\infty$. The rate function
$I_{p}$ is strictly increasing and concave on $\left(0,\infty\right)$
(see Fig. \ref{fig:A-three-dimensionalPlot} for a parametric plot
of $I_{p}$ for different values of the parameters $p\in\left(0,1\right)$). 
\end{cor}
\begin{figure}[h]
\centering \includegraphics[width=0.6\textwidth]{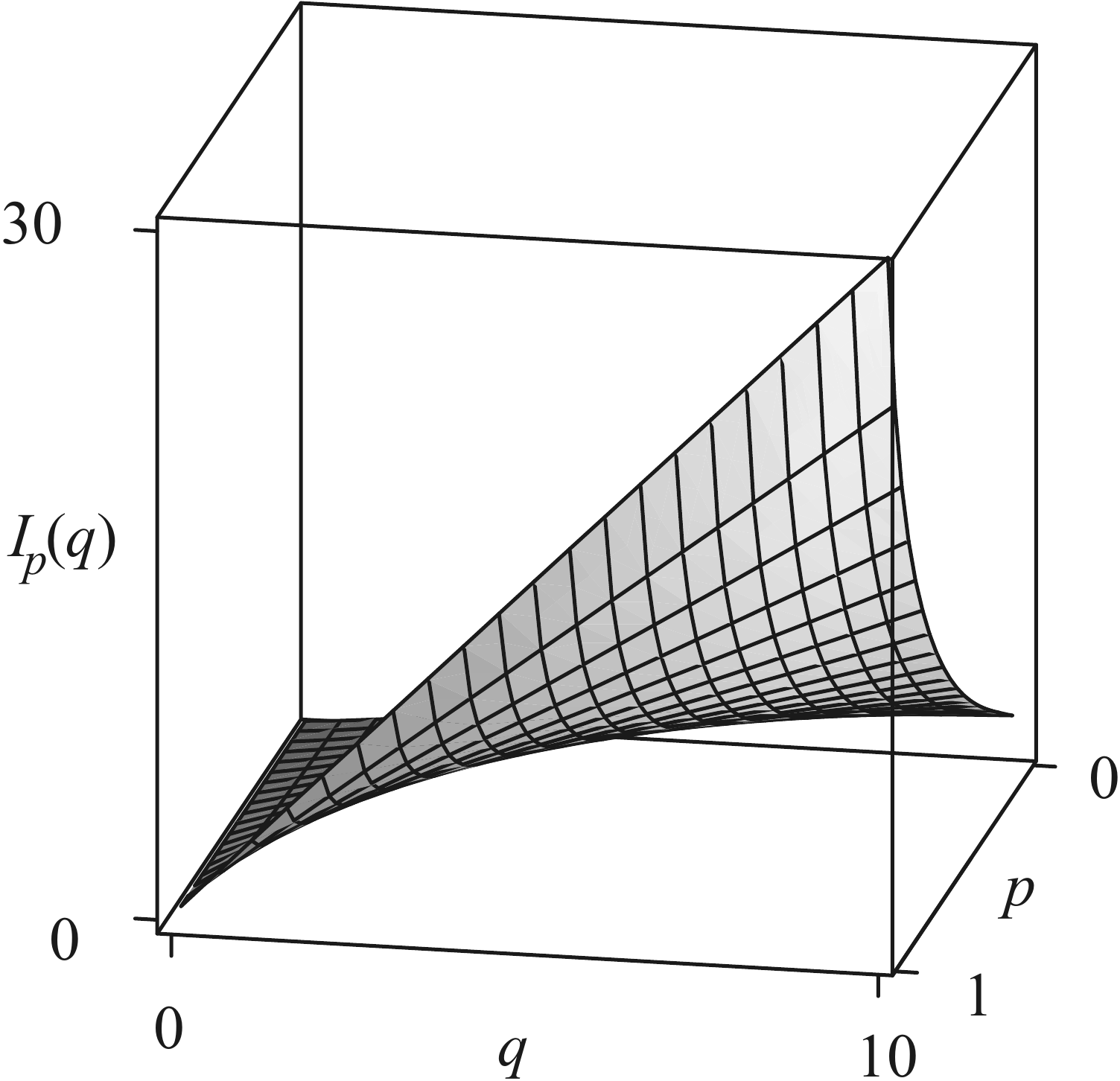}
\caption{\label{fig:A-three-dimensionalPlot}A parametrical plot of the rate
function $I_{p}$ for different parameters $p\in\left(0,1\right)$. }

\label{fig:entropie3D} 
\end{figure}

\begin{rem}
\label{rem:A-prominent-example} At this point we would like to recall
a classical large deviation law for the empirical means of i.i.d.\
random variables. Let $\left(X_{n}\right)$ be a sequence of i.i.d.\
random variables with finite moments and $S_{n}:=\sum_{k=1}^{n}X_{k}$.
Then by the weak law of large numbers we have that the empirical mean
$S_{n}/n$ converges in probability to $\E\left(X_{1}\right)=:\mu$
as $n$ tends to infinity. In this situation a large deviation principle
holds with convex rate function $I$, that is \[
\limsup\limits _{n}\frac{1}{n}\log\mathbb{P}\left(\frac{X_{n}}{n}>t\right)=-I(t),\quad t>\mu.\]
In fact, $I$ is given by the Legendre transform of the Helmholtz
free energy function \[
H_{X}:t\mapsto\mathbb{E}\left(\exp\left(t\cdot X_{1}\right)\right),\]
which is necessary convex. Since in our situation the rate function
is not convex it cannot be the Legendre transform of a free energy
function. More precisely, if we set $w_{n}:=nz_{n+1}/Z_{n}$ then
a possible candidate for the free Helmholtz energy would be \[
H_{w}:t\mapsto\lim_{n\to\infty}\frac{1}{n}\log\mathbb{E}\left(\exp\left(t\cdot w_{n}\right)\right).\]
Since $I_{p}$ is not the Legendre transform of $H_{w}$ we conclude
by a general Large Deviation Theorem of Ellis \cite{MR793553} that
$H$ does not define a finite and differentiable function (for a generalisation
see also \cite{MR1819804}). 
\end{rem}
\renewcommand{\thesubfigure}{{\rm(\alph{subfigure})}}
\captionsetup[subfigure]{labelformat=simple, labelsep=space, listofformat=subsimple}

\begin{figure}[!htbp]
\subfloat[$p=0.1$ ]{\includegraphics[width=0.47\textwidth]{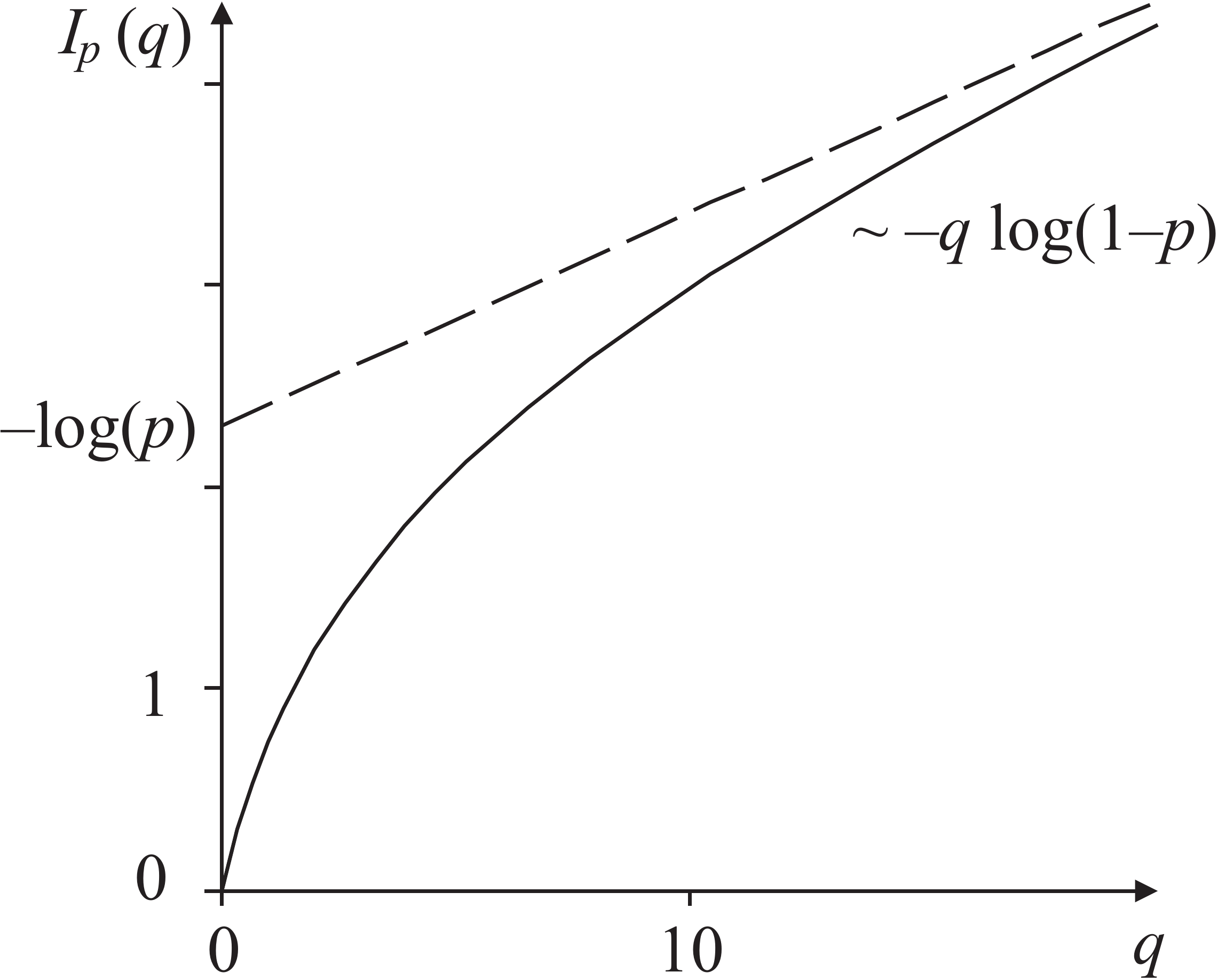}
} \hspace*{\fill}\subfloat[\label{fig:asmpLineare}$p=0.3$ ]{\includegraphics[width=0.47\textwidth]{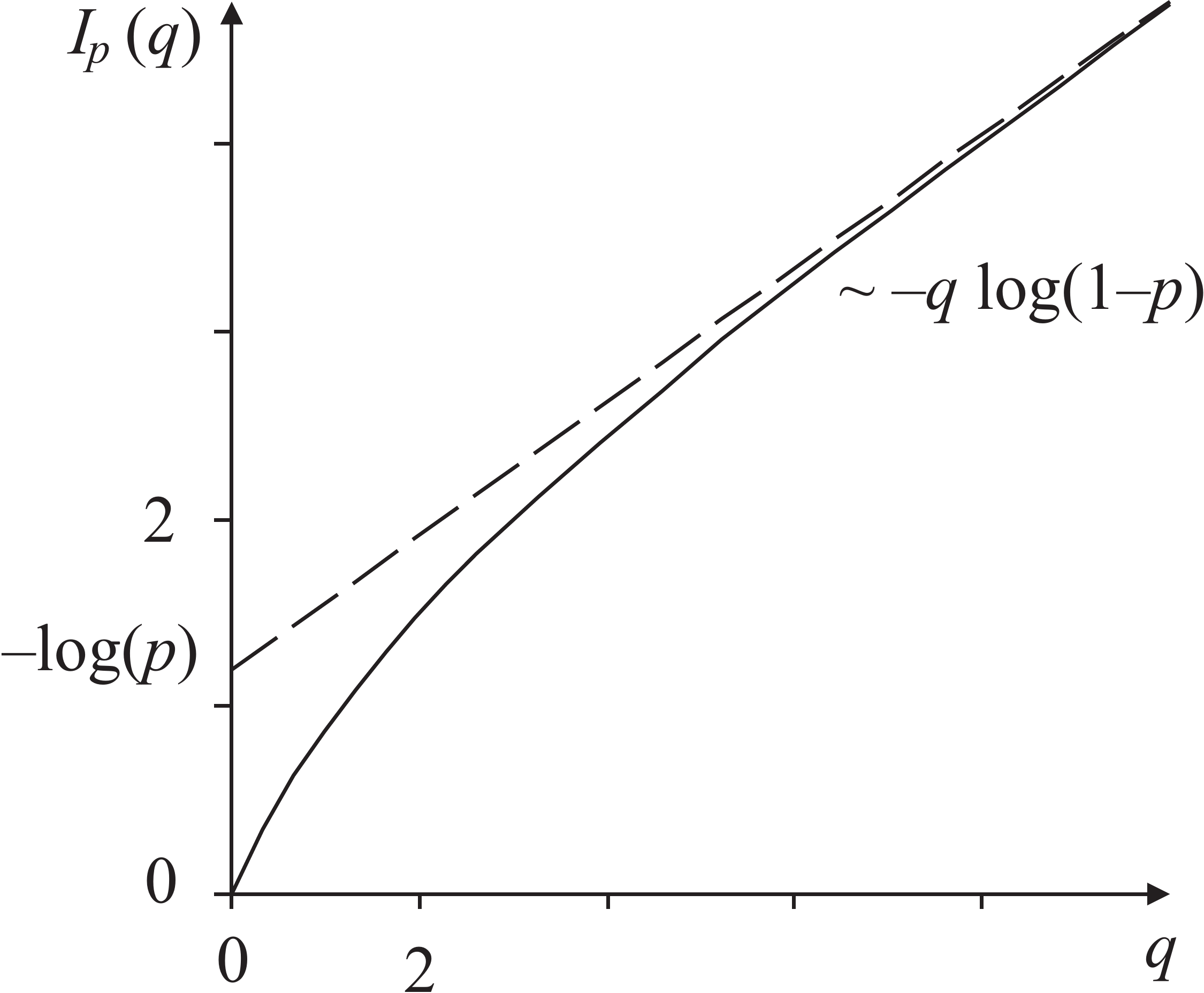}

}\caption{\label{fig:Entropiefunktionen} Graphs of the rate functions for different
values of $p$. }

\end{figure}

\begin{rem}
Furthermore, Fig.\ \ref{fig:Entropiefunktionen} might suggest that
the rate function $I_{p}$ approaches an asymptotic line for $q$
tending to infinity. Indeed, since the parameter $C_{p}\left(q\right)=\left((1-p)^{-(1+q)}-1\right)^{-1}$
vanishes for $q\to\infty$, a straight forward calculation gives\begin{eqnarray*}
I_{p}(q) & \sim & -q\cdot\ln(1-p)-\ln(p)\end{eqnarray*}
for $q$ tending to infinity. That is, the asymptotic line has slope
$-\log\left(1-p\right)$ and intersects the $I_{p}$-axis in $-\log\left(p\right)$
(see Fig.\ \ref{fig:Entropiefunktionen}\subref{fig:asmpLineare}).
\end{rem}
\begin{rem}
We would also like to remark that the sets $\left\{ z_{n+1}/Z_{n}>q\right\} $
considered in this paper play a crucial rôle in the fractal analysis
of the set of non-differentiability of certain singular distribution
functions (Devil's staircases) (cf.\ \cite{Li:07,LiXiaoDekking:02};
for another account on the connection between fractal geometry and
large deviations we refer to \cite{MR1948295}). 
\end{rem}
The proof of Theorem \ref{thm:AsympMain} heavily depends on a generalised
version of Laplace's Method. Since this statement might be of some
interest for itself we will give the precise statement at this point. 
\begin{prop}
[Generalised Laplace's Methode]\label{thm:Laplace} For $n\in\N$,
let $F_{n}\in\mathcal{C}^{3}(\mathbb{R})$ be a concave function having
a unique absolute maximum in $x_{n}\in\mathbb{R}$ such that $F''_{n}(x_{n})<0$.
Also suppose that there exists a sequence $\left(r_{n}\right)_{n\in\mathbb{N}}\in\mathbb{R}_{+}^{\mathbb{N}}$
such that for \[
\sigma_{n}:=\sqrt{\frac{1}{-F''_{n}(x_{n})}}\qquad\mbox{and}\qquad\eta_{n}:=\sup\left\{ \left|F'''_{n}(\xi)\right|:\xi\in(x_{n}-r_{n},x_{n}+r_{n})\right\} \]
 we have, as $n$ tends to infinity, \[
\frac{r_{n}}{\sigma_{n}}\to\infty\qquad\mbox{and}\qquad\delta_{n}:=\sigma_{n}^{2}r_{n}\eta_{n}\to0.\]
Let $I\subset\mathbb{R}$ be an interval such that for all $n\in\N$
large enough we have $(x_{n}-r_{n},x_{n}+r_{n})\subset I$. Then \[
\int_{I}\e^{F_{n}(x)}dx\sim\sqrt{\frac{2\pi}{-F''_{n}(x_{n})}}\,\e^{F_{n}(x_{n})}.\]

\end{prop}
See Remark \ref{rem:OneSided} for corresponding one-sided versions
of this theorem.

The paper is organised as follows. In Section \ref{sec:Proofs-of-theLD}
we give a proof of Corollary \ref{cor:MainTheorm} and Theorem \ref{thm:AsympMain}
and in Section \ref{chapter:allgmeinLM} we will prove the generalised
Laplace's Method. The Appendix is devoted to some special functions
and their basic properties which are crucial in the proof of Theorem
\ref{thm:AsympMain}.

\section{Proof of the large deviation laws\label{sec:Proofs-of-theLD}}

\subsection{Proof of Corollary \ref{cor:MainTheorm}}
\begin{proof}
[Proof of  Corollary \ref{cor:MainTheorm}] The fact that $I_{p}$
is the rate function as claimed in the corollary follows immediately
form Theorem \ref{thm:AsympMain}. The first derivative of $I_{p}$
is given by \begin{equation}
I'_{p}(q)=-\left(\frac{1}{\left(1-p\right)^{-(1+q)}-1}+1\right)\ln\left(1-p\right),\label{eq:firstDerOfI}\end{equation}
which is strictly positive on $\left(0,\infty\right)$. Therefore,
the rate function is strictly increasing on $(0,\infty)$. For the
second derivative we have \begin{equation}
I''_{p}(q)=-\left(\frac{1}{1-p}\right)^{(1+q)}\left(\frac{\ln\left(1-p\right)}{\left(1-p\right)^{-(1+q)}-1}\right)^{2}\label{eq:SecDerivOfI}\end{equation}
which is strictly negative on $\left(0,\infty\right)$ showing the
concavity of $I_{p}$.
\end{proof}

\subsection{Proof of Theorem \ref{thm:AsympMain}}

Theorem \ref{thm:AsympMain} is an immediate consequence of the following
lemmas and Theorem \ref{thm:Laplace}. For the definitions of the
Beta Function $B$ and the Polygamma Function $\psi_{k}$, $k\in\N_{0}$,
we refer the reader to the Appendix of this paper. 
\begin{lem}
\label{lem:Probability} For the large deviation probability we have
\[
\mathbb{P}\left(\omega\Big|\frac{z_{n+1}(\omega)}{Z_{n}(\omega)}>{q}\right)\asymp\sum\limits _{x=0}^{\infty}g(n)e^{f_{n}(x)},\]
 where $f_{n}:x\mapsto-\ln B(x+1,n-1)+x(1+q)\ln(1-p)$ and $g_{n}:=\frac{p^{n}(1-p)^{nq}}{n-1}$,
$n\in\N$ and the comparability constants can be chosen to be $1-p$
and $1$. \end{lem}
\begin{proof}
For $x\in\R$ let $\lfloor x\rfloor$ denote the largest integer not
exceeding $x$. Using the definition (\ref{eq:BetaDef}) of the Beta
Function as introduced in the Appendix we have \begin{eqnarray*}
\mathbb{P}\left(\frac{z_{n+1}}{Z_{n}}>q\right) & = & \sum\limits _{k=n}^{\infty}\left({k-1 \choose n-1}p^{n}(1-p)^{k-n}\left(\sum\limits _{l=\lfloor kq\rfloor+1}^{\infty}(1-p)^{l-1}p\right)\right)\\
 & = & \sum\limits _{k=n}^{\infty}\frac{(k-1)!}{(n-1)!(k-n)!}p^{(n+1)}(1-p)^{k-n}\left(\sum\limits _{l=\lfloor kq\rfloor+1}^{\infty}(1-p)^{l-1}\right)\\
 & = & \sum\limits _{k=n}^{\infty}\frac{(k-1)!}{(n-1)!(k-n)!}p^{(n+1)}(1-p)^{k-n}\frac{(1-p)^{\lfloor kq\rfloor}}{p}\\
 & = & \sum\limits _{k=0}^{\infty}\frac{(k+n-1)!}{(n-1)!(k)!}p^{n}(1-p)^{k+\lfloor(k+n)q\rfloor}\\
 & \asymp & \sum\limits _{k=0}^{\infty}\frac{(k+n-1)!}{(n-1)!(k)!}p^{n}(1-p)^{k+(k+n)q}\\
 & = & \sum\limits _{k=0}^{\infty}\frac{\Gamma(k+n)}{(n-1)\Gamma(n-1)\Gamma(k+1)}p^{n}(1-p)^{k+q(k+n)}\\
 & = & \sum\limits _{k=0}^{\infty}\frac{1}{(n-1)B(k+1,n-1)}p^{n}(1-p)^{k+q(k+n)}\\
 & = & \sum\limits _{k=0}^{\infty}\frac{p^{n}(1-p)^{nq}}{n-1}e^{-\ln B(k+1,n-1)+k(1+q)\ln(1-p)}=\sum\limits _{k=0}^{\infty}g_{n}e^{f_{n}(k)}.\end{eqnarray*}
In here, the comparability constants can be chosen to be $1-p$ and
$1$. To verify Remark \ref{rem:qinNSimResult} we notice that for
$q\in\N$ the comparability may in fact be replaced by equality. This
proves the Lemma.\end{proof}
\begin{lem}
\label{lem:Maximum} For $n>2$ we have that for the maximum $x_{n}$
of \[
f_{n}:\R_{+}\to\R;\; x\mapsto-\ln B(x+1,n-1)+x(1+q)\ln(1-p)\]
 we have $x_{n}\sim Cn$ with $C=C_{p}\left(q\right)=\left(\left(1-p\right)^{-q-1}-1\right)^{-1}$
as defined in the introduction. \end{lem}
\begin{proof}
Using the relations (\ref{eq:BetaDef}), (\ref{eq:Polygamma}) and
(\ref{eq:PolyGammaRekursion}) from the Appendix, we have for $x>0$
\begin{eqnarray*}
f'_{n}(x) & = & \big((-\ln B(x+1,n-1)+x(1+{q})\ln(1-p)\big)'\\
 & = & -\frac{\left(\partial B/\partial x\right)(x+1,n-1)}{B(x+1,n-1)}-\underbrace{(1+q)\ln(\frac{1}{1-p})}_{=:C_{0}}\\
 & = & -\left(\frac{\Gamma'(x+1)}{\Gamma(x+1)}-\frac{\Gamma'(x+n)}{\Gamma(x+n)}\right)-C_{0}\\
 & = & \psi_{0}(x+n)-\psi_{0}(x+1)-C_{0}=\sum\limits _{y=1}^{n-1}\frac{1}{x+y}-C_{0},\end{eqnarray*}
where $C_{0}=-(1+q)\ln(1-p)>0$. Now uniformly in $x>0$ we have \[
\sum\limits _{y=1}^{n-1}\frac{1}{x+y}\sim\int_{1}^{n-1}\frac{1}{x+y}\, dy=\ln\left(1+\frac{n-2}{x-1}\right).\]
Hence, the root $y_{n}:=C\left(n-2\right)-1$ of $x\mapsto\log\left(1+\left(n-2\right)/\left(x-1\right)\right)-C_{0}$
is asymptotically equal to the root of $f_{n}'$. 

Using (\ref{eq:Polygamma}) and (\ref{eq:PolyGammaRekursion}) we
get for the second derivative of $f_{n}$ in $x>0$ \begin{eqnarray*}
f''(x) & = & \psi_{1}(x+n)-\psi_{1}(x+1)=\sum\limits _{k=1}^{n-1}-\frac{1}{\left(x+k\right)^{2}}<0.\end{eqnarray*}
This proves the lemma.\end{proof}
\begin{lem}
\label{lem:Asymp} Let $r_{n}:=x_{n}^{\beta}$ with $\beta\in(1/2,1)$,
$\sigma_{n}:=1/\sqrt{-f''_{n}(x_{n})}$, and $x_{n}\sim Cn$, $n\in\N$.
Then, for $n$ tending to infinity, we have \[
\sigma_{n}\sim\sqrt{n}\cdot\sqrt{C^{2}+C}\qquad\mbox{ and }\qquad\frac{r_{n}}{\sigma_{n}}\to\infty.\]
\end{lem}
\begin{proof}
Using the explicit formula for the second derivative of $f_{n}$ and
the fact that $x_{n}\sim Cn$ we get \begin{eqnarray*}
\sigma_{n} & = & \frac{1}{\sqrt{-f''_{n}(x_{n})}}\sim\left(\int_{1}^{n-1}\frac{1}{(x_{n}+y)^{2}}dy\right)^{-1/2}=\left(\frac{2-n}{(x_{n}+n-1)(x_{n}+1)}\right)^{-1/2}\\
 & \sim & \sqrt{n}\cdot\sqrt{C^{2}+C}.\end{eqnarray*}
In the same way we get for $n$ tending to infinity \begin{eqnarray*}
\frac{r_{n}}{\sigma_{n}} & = & x_{n}^{\beta}\sqrt{-f''_{n}(x_{n})}=x_{n}^{\beta}\sqrt{\frac{n}{x_{n}(x_{n}+n)}}\sim\frac{C^{\beta}}{\sqrt{C^{2}+C}}n^{\beta-1/2}\to\infty.\end{eqnarray*}
 \end{proof}
\begin{lem}
\label{chapter:returntimeLEM4} Let $r_{n}:=x_{n}^{\beta}$ with $\beta\in(1/2,1)$,
$n\in\N$, and $\mathcal{U}_{n}:=\left(x_{n}-r_{n},x_{n}+r_{n}\right)$
with $x_{n}\sim Cn$. Then, for $n$ tending to infinity, we have
\[
\eta_{n}:=\sup\left\{ \left|f'''_{n}(\xi)\right|:\xi\in\mathcal{U}_{n}\right\} \asymp\frac{1}{n^{2}}\qquad\mbox{and}\qquad\delta_{n}:=\sigma_{n}^{2}r_{n}\eta_{n}\to0.\]
\end{lem}
\begin{proof}
For the third derivative of $f_{n}$ we have by (\ref{eq:Polygamma})
and (\ref{eq:PolyGammaRekursion}) of the Appendix uniformly in $\mathcal{U}_{n}$
\begin{eqnarray*}
f'''_{n}(x) & = & \psi_{2}(x+n)-\psi_{2}(x+1)=\sum\limits _{y=x+1}^{x+n-1}\frac{2}{y^{3}}\sim\int\limits _{1}^{n-1}\frac{2}{(x+y)^{3}}dy\\
 & = & \frac{1}{(x+1)^{2}}-\frac{1}{(x+n-1)^{2}}.\end{eqnarray*}
The maximum of $\left|f_{n}'''\right|$ on the interval $\mathcal{U}_{n}$
is obtained in the left boundary point of this interval and since
$\beta<1$ we obtain \[
\eta_{n}:=\sup\limits _{\xi\in\mathcal{U}_{n}}\left|F'''_{n}(\xi)\right|\asymp\frac{1}{(x_{n}-r_{n}+1)^{2}}-\frac{1}{(x_{n}-r_{n}+n-1)^{2}}\asymp\frac{1}{n^{2}}.\]
Hence, for $n\to\infty$, we have \[
\delta_{n}=r_{n}\sigma_{n}^{2}\eta_{n}\asymp\frac{r_{n}\sigma_{n}^{2}}{n^{2}}=\frac{x_{n}^{\beta+1}(x_{n}+n)}{n^{3}}\sim C^{1+\beta}(C+1)n^{\beta-1}\to0.\]

\end{proof}
Now we are in the position to give a proof of our main theorem.
\begin{proof}
[Proof of  Theorem \ref{thm:AsympMain}] Combining the above lemmas
for $x_{n}$ the unique maximum of $f_{n}$, $r_{n}:=x_{n}^{\beta}$,
$n\in\N,$ for some $\beta\in(1/2,1)$, we deduce from Theorem \ref{thm:Laplace}
and Stirling's Formula ($n!\sim\sqrt{2\pi n}\left(n/\e\right)^{n}$)
that\begin{eqnarray*}
\int_{0}^{\infty}g_{n}\e^{f_{n}(x)}\, dx & \sim & g_{n}\e^{f_{n}(x_{n})}\sqrt{\frac{2\pi}{-f''_{n}(x_{n})}},\\
 & \sim & \sqrt{2\pi Cn(C+1)}\frac{(x_{n}+n-1)!}{(n-1)!(x_{n})!}p^{n}(1-p)^{x_{n}+\left((x_{n}+n)q\right)}\\
 & \sim & \left(\frac{Cn+n-1}{n-1}\right)^{(C+1)n}\left(\frac{n-1}{Cn}\right)^{Cn}p^{n}(1-p)^{Cn(1+q)+nq}.\end{eqnarray*}
Since,\[
\left(\frac{Cn+n-1}{n-1}\right)^{(C+1)n}\sim\e^{C}\left(C+1\right)^{(C+1)n}\qquad\mbox{and }\qquad\left(\frac{n-1}{Cn}\right)^{Cn}\sim\e^{-C}\left(\frac{1}{C}\right)^{Cn}\]
we get\begin{eqnarray*}
\int_{0}^{\infty}g_{n}\e^{f_{n}(x)}\, dx & \sim & \left(\frac{(C+1)^{C+1}}{C^{C}}p(1-p)^{C(1+q)+q}\right)^{n}.\end{eqnarray*}
Due to Lemma \ref{lem:Probability} we are now left to show that \begin{equation}
\sum\limits _{k=0}^{\infty}\e^{f_{n}(k)}\sim\int_{0}^{\infty}\e^{f_{n}(x)}\, dx.\label{eq:SumsimInt}\end{equation}
Since $f_{n}$ is concave we have that $f_{n}$ is increasing on $\left(0,x_{n}\right)$
and decreasing on $\left(x_{n},\infty\right)$. Consequently, \[
\sum\limits _{k=0}^{\left\lfloor x_{n}\right\rfloor -1}\e^{f_{n}(k)}<\int_{0}^{x_{n}}\e^{f_{n}(x)}\, dx<\sum\limits _{x=1}^{\left\lfloor x_{n}\right\rfloor +1}\e^{f_{n}(x)},\]
 and \[
\sum\limits _{k=\left\lfloor x_{n}\right\rfloor +1}^{\infty}\e^{f_{n}(k)}<\int_{x_{n}}^{\infty}\e^{f_{n}(x)}\, dx<\sum\limits _{k=\left\lfloor x_{n}\right\rfloor }^{\infty}\e^{f_{n}(k)}.\]
This gives \[
\left|\sum\limits _{k=0}^{\infty}\e^{f_{n}(k)}-\int_{0}^{\infty}\e^{f_{n}(x)}\, dx\right|\leq\e^{f_{n}(0)}+\e^{f_{n}(\left\lfloor x_{n}\right\rfloor )}\leq2\e^{f_{n}(x_{n})}.\]
 Since by Theorem \ref{thm:Laplace} and Lemma \ref{lem:Asymp} \[
\int_{0}^{\infty}\e^{f_{n}(x)}\, dx\sim\e^{f_{n}(x_{n})}\sqrt{\frac{2\pi}{-f''_{n}(x_{n})}}\asymp\e^{f_{n}(x_{n})}\sqrt{n}\]
the asymptotic in (\ref{eq:SumsimInt}) follows.
\end{proof}

\section{Proof of the generalised Laplace's Method\label{chapter:allgmeinLM} }

\begin{proof}
[Proof of Theorem \ref{thm:Laplace}] Using Taylor's formula we have
for every $x\in(x_{n}-r_{n},x_{n}+r_{n})=:\mathcal{U}_{n}$ that there
exists $\xi_{x}\in\mathcal{U}_{n}$ such that \begin{equation}
F_{n}(x)=\underbrace{F_{n}(x_{n})+\frac{F_{n}^{''}(x_{n})}{2!}(x-x_{n})^{2}}_{=:G_{n}\left(x\right)}+\frac{F_{n}^{'''}(\xi_{x})}{3!}(x-x_{n})^{3}.\label{eq:Taylor}\end{equation}
 The second order Taylor approximation is asymptotically exact on
$\mathcal{U}_{n}$ in the sense that, for $n\to\infty$, \begin{eqnarray*}
\sup\limits _{x,\xi\in\mathcal{U}_{n}}\left|\frac{F_{n}^{'''}(\xi)\left(x-x_{n}\right)^{3}/3!}{F_{n}^{''}(x_{n})\left(x-x_{n}\right)^{2}/2!}\right| & = & \frac{1}{3}\sup\limits _{\xi\in\mathcal{U}_{n}}r_{n}\sigma_{n}^{2}\left|F_{n}^{'''}(\xi)\right|=\frac{1}{3}\delta_{n}\to0.\end{eqnarray*}
This implies that for the integration restricted to $\mathcal{U}_{n}$
we have \begin{eqnarray*}
\int\limits _{\mathcal{U}_{n}}\e^{F_{n}(x)}dx & = & \int\limits _{\mathcal{U}_{n}}\exp\left(G_{n}\left(x\right)+\frac{F_{n}^{'''}(\xi_{x})}{3!}(x-x_{n})^{3}\right)\, dx\\
 & = & \e^{F_{n}(x_{n})}\int\limits _{\mathcal{U}_{n}}\exp\left(-\frac{\left(x-x_{n}\right)^{2}}{2\sigma_{n}^{2}}\left(1+\frac{F_{n}^{'''}(\xi_{x})\left(x-x_{n}\right)^{3}}{3F_{n}^{''}(x_{n})\left(x-x_{n}\right)^{2}}\right)\right)\, dx\\
 & \sim & \e^{F_{n}(x_{n})}\sigma_{n}\int\limits _{-r_{n}/\sigma_{n}}^{r_{n}/\sigma_{n}}\exp\left(-\frac{t^{2}}{2}\right)\, dt\sim\e^{F_{n}(x_{n})}\sigma_{n}\sqrt{2\pi}.\end{eqnarray*}
Hence, we are left to show that, for $n\to\infty$, we have \[
\frac{1}{\sigma_{n}\e^{F_{n}(x_{n})}}\int\limits _{\mathbb{R}\setminus\mathcal{U}_{n}}\e^{F_{n}(x)}\, dx\to0.\]
The error term in the second order Taylor approximation on $\mathcal{U}_{n}$
is at most \[
d_{n}:=\frac{r_{n}^{3}}{3}\sup\left\{ \left|F'''_{n}(\xi)\right|:\xi\in\mathcal{U}_{n}\right\} =\frac{\delta_{n}r_{n}^{2}}{3\sigma_{n}^{2}}.\]
To find an upper bound for the integrand on the complement of $\mathcal{U}_{n}$
we will consider two straight lines connecting the point $\left(x_{n},F_{n}\left(x_{n}\right)\right)$
with the two points $\left(x_{n}\pm r_{n},G_{n}\left(x_{n}\pm r_{n}\right)+d_{n}\right)$.
 The absolut value of the slopes of the two lines is then given by
$r_{n}/\sigma_{n}^{2}\left(1/2+\delta_{n}/3\right)$ and, since $F_{n}$
is concave, on $\R\setminus\mathcal{U}_{n}$ these lines lie necessarily
above the graph of $F_{n}$. Hence, for $n$ tending to infinity,
we obtain the following upper bound for the integral. \begin{eqnarray*}
\frac{1}{\sigma_{n}\e^{F_{n}(x_{n})}}\int\limits _{\R\setminus\mathcal{U}_{n}}\e^{F_{n}(x)}\, dx & \leq & 2\sigma_{n}^{-1}\int\limits _{r_{n}}^{\infty}\e^{-t\left(r_{n}/\sigma_{n}^{2}\left(1/2+\delta_{n}/3\right)\right)}\, dt\\
 & = & \frac{2\sigma_{n}/r_{n}}{\left(1/2+\delta_{n}/3\right)}\e^{-r_{n}^{2}/\sigma_{n}^{2}\left(1/2+\delta_{n}/3\right)}\to0.\end{eqnarray*}
Combining all of the above, we therefore obtain\begin{eqnarray*}
\int_{\mathbb{R}}\e^{F_{n}(x)}dx & = & \int\limits _{\mathcal{U}_{n}}\e^{F_{n}(x)}\, dx+\int\limits _{\mathbb{R}\setminus\mathcal{U}_{n}}\e^{F_{n}(x)}\, dx\\
 & = & \int\limits _{x_{n}-r_{n}}^{x_{n}+r_{n}}\e^{F_{n}(x)}\, dx\left(1+\frac{\int\limits _{\mathbb{R}\setminus\mathcal{U}_{n}}\e^{F_{n}(x)}\, dx}{\int\limits _{x_{n}-r_{n}}^{x_{n}+r_{n}}\e^{F_{n}(x)}\, dx}\right)\\
 & \sim & e^{F_{n}(x_{n})}\sqrt{2\pi}\sigma_{n}\left(1+\frac{\int\limits _{\mathbb{R}\setminus\mathcal{U}_{n}}\e^{F_{n}(x)}\, dx}{\e^{F_{n}(x_{n})}\sqrt{2\pi}\sigma_{n}}\right)\sim\sqrt{2\pi}\sigma_{n}\e^{F_{n}(x_{n})}.\end{eqnarray*}
\end{proof}
\begin{rem}
\label{rem:OneSided}If the condition $\mathcal{U}_{n}\subset I$
is not fulfilled for all sufficient large $n\in\N$, then a one-sided
version of Theorem \ref{thm:Laplace} still may hold. \end{rem}
\begin{enumerate}
\item \emph{If $I=[a,\infty)$ and $\frac{x_{n}-a}{\sigma_{n}}\stackrel{n\rightarrow\infty}{\rightarrow}0$,
then we have \[
\int_{I}\e^{F_{n}(x)}dx\sim\e^{F_{n}(x_{n})}\sqrt{\frac{\pi}{-2F''_{n}(x_{n})}}.\]
 }
\item \emph{If $I=(\infty,b]$ and $\frac{x_{n}-b}{\sigma_{n}}\stackrel{n\rightarrow\infty}{\rightarrow}0$,
then we have \[
\int_{I}\e^{F_{n}(x)}dx\sim\e^{F_{n}(x_{n})}\sqrt{\frac{\pi}{-2F''_{n}(x_{n})}}.\]
}
\end{enumerate}
For the proof of these claims one has to consider the one-sided (say
right-sided for case (1)) Taylor expansion of $F_{n}$ and mainly
use the elementary fact that with $h_{n}:x\mapsto\frac{1}{\sigma_{n}\sqrt{2\pi}}\exp\left(-\frac{1}{2}\left(\frac{x-x_{n}}{\sigma_{n}}\right)^{2}\right)$
we have under the assumptions in (1) \[
\int_{a}^{a+r_{n}}h_{n}(x)\, dx\sim\frac{1}{2}.\]

\section{Appendix}

\textit{\emph{For a good account on the following facts we refer the
reader to \cite{abramowitz+stegun}. }}

\textit{\emph{For positive $x\in\R$ the }}\textit{Gamma Function
}\textit{\emph{is given by}} \begin{equation}
\Gamma(x):=\int_{0}^{\infty}t^{x-1}\e^{-t}dt,\label{eq:GammaDef}\end{equation}
\textit{\emph{and for positive $x$ and $y$  the}}\textit{ Beta Function
}\textit{\emph{is given by}} \begin{equation}
B(x,y):=\frac{\Gamma(x)\Gamma(y)}{\Gamma(x+y)}.\label{eq:BetaDef}\end{equation}
\textit{\emph{For the partial derivative of the Beta Function we have}}
\begin{equation}
\frac{\partial}{\partial x}B(x,y)=B(x,y)\left(\frac{\Gamma'(x)}{\Gamma(x)}-\frac{\Gamma'(x+y)}{\Gamma(x+y)}\right)=B(x,y)(\psi_{0}(x)-\psi_{0}(x+y))\label{eq:Digamma}\end{equation}
\textit{\emph{where $\psi_{n}$ denotes the}}\textit{ Polygamma Function
}\textit{\emph{defined, for $n\in\N$, by}} \begin{equation}
\psi_{n}(x):=\frac{\partial^{n+1}}{\partial x^{n+1}}\ln\Gamma(x).\label{eq:Polygamma}\end{equation}
The following recursion formula holds \begin{equation}
\psi_{n}(x+1)=\psi_{n}(x)+\frac{(-1)^{n}n!}{x^{n+1}}.\label{eq:PolyGammaRekursion}\end{equation}

\end{document}